\documentclass[12pt,letterpaper]{amsart}
\usepackage{euler, epic,eepic,latexsym, amssymb, amscd, amsfonts, xypic, url, color, epsfig}
\usepackage{MnSymbol}
\input xy
\xyoption{all}

 \newlength{\baseunit}               % the basic unit length
 \newcount{\numlines}                % depth of picture (in number of lines)
\alph{subsection}

\newtheorem*{tmnl}{Theorem}
\newtheorem*{conj}{Conjecture}
\newtheorem{tm}{Theorem}
\newtheorem{pr}[tm]{Proposition}
\newtheorem{lm}[tm]{Lemma}
\newtheorem{co}[tm]{Corollary}

\theoremstyle{definition}
\newtheorem{df}[tm]{Definition}

\theoremstyle{remark}
\newtheorem{rmk}[tm]{Remark}

\newcounter{nootje}
\newcommand{\calC}{{ \mathcal C}}

\newcommand{\calJ}{{ \mathcal J}}

\newcommand{\calM}{{ \mathcal M}}

\newcommand{\calO}{{ \mathcal O}}

\newcommand{\ODF}{{\mathcal{O}(\mathcal{D}(\phi))}}
\newcommand{\Jb}{\overline{\mathcal{J}}}
\newcommand{\DR}{\operatorname{DR}}

\begin{document}
\pagestyle{plain}
\title{Extending the Double Ramification Cycle using Jacobians}

%\date{Sunday May 26, 2013.}
\author{David Holmes}
\address{Mathematisch Instituut Leiden, Leiden, NL}
\email{holmesdst@math.leidenuniv.nl}
\author{Jesse Leo Kass}
\address{Dept. of Mathematics, University of South Carolina, Columbia~SC}
\email{kassj@math.sc.edu}
\author{Nicola Pagani}
\address{Dept. of Mathematical Sciences, University of Liverpool, UK}
\email{pagani@liv.ac.uk}
\begin{abstract}
We prove that the extension of the double ramification cycle defined by the first-named author (using modifications of the stack of stable curves) coincides with one of those defined by the last-two named authors (using an extended Brill--Noether locus on a suitable compactified universal Jacobians). In particular, in the untwisted case we deduce that both of these extensions coincide with that constructed by Li and Graber--Vakil using a virtual fundamental class on a space of rubber maps.

\end{abstract}

\maketitle

{\parskip=12pt % closing bracket is just before the bibliography
%\arabic{subsection}

%%%%%%%%%%%%%%%%%%%%%%%%%%%%%%
\section{Introduction}
%%%%%%%%%%%%%%%%%%%%%%%%%%%%%%

Let $g$ and $n$ be fixed natural numbers with $g,n\geq 1$. Given a fixed nontrivial vector of integers $(k; a_1, \dots, a_n)$  such that $k(2-2g)+ \sum a_i =0$, the \emph{(uncompactified, twisted) double ramification cycle} $\DR \subset \calM_{g, n}$ is defined to be the closed locus of the moduli space $\calM_{g,n}$ of smooth $n$-pointed curves of genus $g$ that consists of those pointed curves $(C, p_1, \dots, p_n)$ such that the line bundle $\omega_C^{- \otimes k}(a_1 p_1 + \dots a_n p_n)$ is trivial.   There are several approaches to extending $\DR$ to a Chow class on $\overline{\calM}_{g, n}$ and then computing this class.

Here we focus on approaches that use the universal Jacobian. Let $\calJ^0_{g,n}$ be the universal Jacobian parameterizing smooth $n$-pointed curves of genus $g$ together with a line bundle of degree zero. The morphism $\sigma\colon \calM_{g,n} \to \calJ^0_{g,n}$ defined by
\begin{equation} \label{sigma}
\sigma ([C, p_1, \ldots, p_n]) = \left[C, p_1, \ldots, p_n, \omega_C^{- \otimes k}(a_1 p_1 + \ldots + a_n p_n)\right]
\end{equation}
is a section of the forgetful morphism $\calJ^0_{g,n} \to \calM_{g,n}$ and the double ramification cycle equals $\sigma^{-1}(E)$, for $E$ the closed locus of $\calJ^0_{g,n}$ that corresponds to the trivial line bundle.

A natural generalization of this approach over $\overline{\calM}_{g,n}$ runs as follows. This time we consider the \emph{multidegree zero universal Jacobian} $\calJ^{\underline{0}}_{g,n}$ (also known in the literature as the \emph{generalized Jacobian}), defined as the moduli stack parameterizing stable $n$-pointed curves of arithmetic genus $g$ together with a line bundle with trivial multidegree (\emph{i.e.}~with degree zero on every irreducible component of every fiber). The stack $\calJ^{\underline{0}}_{g,n}$ still contains the closed locus $E$ that parameterizes trivial line bundles and comes with a  forgetful morphism $p$ to $\overline{\calM}_{g,n}$, but Rule~\eqref{sigma} in general fails to define a morphism and only defines a rational map $\sigma \colon \overline{\calM}_{g,n} \dashedrightarrow \calJ^{\underline{0}}_{g,n}$.

Holmes proposed a way to resolve the indeterminacy of $\sigma$ by modifying the source $\overline{\calM}_{g,n}$. In \cite[Corollary~4.6]{holmes} he constructs a ``minimal'' morphism (see Section~\ref{davidsection}) of normal Deligne--Mumford stacks  $\pi^{\lozenge} \colon \calM^{\lozenge}_{g,n} \to \overline{\calM}_{g,n}$ such that $\pi^{\lozenge  -1}(\calM_{g,n})$ is dense in $\calM^{\lozenge}_{g,n}$ and the rational map $\sigma \circ \pi^{\lozenge}$  extends  (uniquely) to a regular embedding \emph{morphism} \[\sigma^{\lozenge} \colon \calM_{g,n}^\lozenge \to \calJ_{g,n}^{\underline{0}}\times_{\overline{\calM}_{g,n}} \calM_{g,n}^\lozenge.\] Whilst $\pi^{\lozenge}$ is (in general) not proper, Holmes observed that the scheme-theoretic pullback ${\sigma}^{ \lozenge -1}(E)$ \emph{is} proper, so it makes sense to consider the pushforward $\pi^{\lozenge}_*(\sigma^{\lozenge *}[E])$, which we will denote by $[DR^\lozenge]$. When $k=0$ he then proved the equality of Chow classes $[DR^\lozenge]~=~[DR_{LGV}]$,
where the right hand side is the extension of the double ramification cycle to $\overline{\calM}_{g,n}$ due to Li \cite{li01,li02} and Graber--Vakil \cite{graber}. This latter extension is obtained as the pushforward of a certain virtual class defined on a moduli stack $\overline{\calM}_{g,n}(\mathbb{P}^1, \underline{a})^{\sim}$ of rubber maps to $\mathbb{P}^1$, and its class has been computed in terms of standard tautological classes by Janda, Pandharipande, Pixton and Zvonkine, proving an earlier conjecture by Pixton, see \cite{jppz}.

Kass and Pagani proposed another way of resolving the indeterminacy of the rational map $\sigma$ by modifying the target $\calJ^{\underline{0}}_{g,n}$.  In \cite[Section~4]{kasspa2} they constructed, for each nondegenerate $\phi$ in a certain stability vector space $V_{g,n}^0$, a compactified universal Jacobian $\Jb_{g,n}(\phi)$ parameterizing $\phi$-stable rank $1$ torsion-free sheaves on stable pointed curves. They propose extending $E$ to $\Jb_{g,n}(\phi)$ as a Brill--Noether class $w(\phi)$ (a class $w^r_d$ with $d=r=0$).  This produces infinitely many extensions $[DR(\phi)]$, one for each nondegenerate  $\phi \in V_{g,n}^0$, by pulling back $w(\phi)$ along the correspondence induced by the rational map $\sigma(\phi) \colon \overline{\calM}_{g,n} \dashedrightarrow \Jb_{g,n}(\phi)$. See Section~\ref{kasspagani} for more details.

The main result of this paper is
\begin{tmnl} (Theorem~\ref{main})
If  $\phi \in V_{g,n}^0$ is nondegenerate and such that the inclusion $\calJ_{g,n}^{\underline{0}} \subseteq \Jb_{g,n}(\phi)$ holds, we have $[DR(\phi)] = [DR^\lozenge]$.
\end{tmnl}
Recall that when $k=0$, we know by \cite{holmes} that $[DR^\lozenge]= [DR_{LGV}]$, so all three extensions of DR coincide.

We prove the main result in Section~\ref{mainresult} by first showing the equality
\begin{equation} \label{half}
 ([DR^\lozenge] = ) \quad \pi^{\lozenge}_*({\sigma}^{\lozenge *}[E])= p_* ([\Sigma] \cdot [E]) \left(=: \sigma^* ([E]) \right)
\end{equation}
(for $E$ the zero section and $\Sigma$ the Zariski closure of the image of $\sigma$) and then by proving that, when $\phi$ satisfies the hypotheses of the theorem, the Brill--Noether class $w(\phi)$ coincides with the class $[E]$. Equation~\eqref{half} gives a geometric description of the double ramification cycle on $\overline{\calM}_{g,n}$ analogous to the equality $DR = \sigma^{-1}(E)$ on $\calM_{g,n}$, see Remark~\ref{analogous}.

In light of our theorem, \cite[Conjecture~1.4]{holmes} can be reformulated as a relation between the $[DR(\phi)]$'s and Pixton's $k$-twisted cycle $P_g^{g,k} (a_1+ k, \ldots, a_n+k)$ (\cite[Section~1.1]{jppz}).

\begin{conj}If  $\phi \in V_{g,n}^0$ is nondegenerate and such that the inclusion $\calJ_{g,n}^{\underline{0}} \subseteq \Jb_{g,n}(\phi)$ holds, then $[DR(\phi)] = 2^{-g} \cdot P_g^{g,k} (a_1+ k, \ldots, a_n+k)$.
\end{conj}
 This conjecture provides a geometric interpretation of Pixton's cycle as the pull--back via the rational map $\sigma \colon \overline{\mathcal{M}}_{g,n} \dashedrightarrow \Jb_{g,n}(\phi)$ (for $\phi$ as in the theorem) of the class of the zero section $[E]$. % or, equivalently, of the class $w(\phi)$.

In Section~\ref{consequences} we explain our approach to computing all classes $[DR(\phi)]$ (and, in particular, those mentioned in the theorem and in the conjecture). For $\phi \in V_{g,n}^0$ nondegenerate and such that the universal line bundle $\omega_C^{-\otimes k}(a_1p_1+\ldots+a_np_n)$ is $\phi$-stable, the class $[DR(\phi)]$ is computed, by applying cohomology and base change combined with the Grothendieck--Riemann--Roch formula applied to the universal curve, as the top Chern class class of a certain coherent sheaf on $\overline{\mathcal{M}}_{g,n}$. Computing all other $[DR(\phi)]$'s becomes then a matter of keeping track of how they get modified each time a stability hyperplane of $V_{g,n}^0$ is crossed. %$[DR(\phi)]$

%When $k=0$, computing $[DR_{LGV}]$ becomes then a matter of keeping track of how $[DR(\phi)]$ gets modified each time a stability hyperplane of $V_{g,n}^0$ is crossed. Carrying out the same programme for general $k$ could give a geometric interpretation of a cycle that Pixton defined by naturally extending the formula for $[DR_{LGV}]$ to $k \neq 0$.

The double ramification cycle was first computed on the moduli space of curves
of compact type by Hain in \cite{hain13}. Grushevsky--Zakharov extended the calculation to the moduli space
of curves with at most one non-separating node in \cite{grushevsky14}.
Extensions of the double ramification cycle via log geometry were considered in the papers \cite{guere17}
and \cite{marcus17}. The latter supersedes the preprint arXiv:1310.5981, which, for $k=0$, proved the
equality of the double ramification cycles defined via Jacobians and via rubber maps over the locus
of curves of compact type. Another conjectural geometric interpretation of Pixton's  $k$-twisted cycle was given in \cite{farkaspanda} in terms of $k$-twisted canonical divisors.

 %\cite{hain13, grushevsky14, guere17,marcus17}.

Throughout we work over the field $\mathbb{C}$ of complex numbers.

\subsection{Acknowledgements}
DH would like to thank Orsola Tommasi and the third named author for inviting him to a wonderful workshop in G\"oteborg, where part of this work was carried out.

JLK was  supported by a grant from the Simons Foundation  (Award Number 429929) and by the National Security Agency under Grant Number H98230-15-1-0264.  The United States Government is authorized to reproduce and distribute reprints notwithstanding any copyright notation herein. This manuscript is submitted for publication with the understanding that the United States Government is authorized to reproduce and distribute reprints.

NP was supported by the EPSRC First Grant EP/P004881/1 with title  ``Wall-crossing on universal compactified Jacobians''. He would like to thank Alexandr Buryak for the useful feedback on an early version of this preprint. He is grateful to Paolo Rossi  for having initiated him to the DRC and for the invitation to a stimulating workshop in Dijon, and to Nicola Tarasca for several helpful discussions.

We thank the anonymous referee for his rapid and thorough review.
%%%%%%%%%%%%%%%%%%%%%%%%%%%%
\section{Background}
%%%%%%%%%%%%%%%%%%%%%%%%%%

%%%%%%%%%%%%%%%%%%%%
\subsection{Review of Holmes' work on extending the Abel--Jacobi section} \label{Section: HolmesWork}
%%%%%%%%%%%%%%%%%%%%

\label{davidsection}

Here we recall the definition of the \emph{universal $\sigma$-extending stack} $\calM^{\lozenge}_{g,n}$ over $\overline{\calM}_{g,n}$; note that this space depends on the vector $(k; a_1, \dots, a_n)$.

\begin{df} \label{sigmaext} We call a morphism $t\colon T \to \overline{\calM}_{g,n}$ from a normal Deligne--Mumford stack \emph{$\sigma$-extending} if $t^{-1}\calM_{g,n}$ is dense in $T$, and if the induced rational map $\sigma_T\colon T \dashrightarrow  \calJ^{\underline{0}}_{g,n}$ extends (necessarily uniquely) to a morphism $T \to  \calJ^{\underline{0}}_{g,n}$.  We define the \emph{universal $\sigma$-extending stack} $\calM^{\lozenge}_{g,n}$ to be the terminal object in the category of $\sigma$-extending morphisms to $\overline{\calM}_{g,n}$. \end{df}

The existence of a terminal object $\pi^{\lozenge} \colon \calM^{\lozenge}_{g,n} \to \overline{\calM}_{g,n}$ was established in \cite[Theorem~3.15]{holmes}, where $\pi^{\lozenge}$ was also shown to be representable by algebraic spaces, separated and birational (more precisely, an isomorphism over the locus of compact type curves). Furthermore, $\calM^{\lozenge}_{g,n}$ is naturally equipped with a log structure making it log \'etale over $\overline{\calM}_{g,n}$ (the latter comes with a natural log structure, called \emph{basic log structure}, from \cite{kato00}).

From Definition~\ref{sigmaext} we deduce the existence of a regular embedding \[\sigma^{\lozenge} \colon \calM_{g,n}^\lozenge \to \calJ_{g,n}^{\lozenge}:= \calJ_{g,n}^{\underline{0}}\times_{\overline{\calM}_{g,n}} \calM_{g,n}^\lozenge\] extending the rational section $\sigma\colon \overline{\calM}_{g,n} \dashedrightarrow \calJ_{g,n}^{\underline{0}}$. Writing $E$ for the schematic image of the zero section in $\calJ_{g,n}^{\lozenge}$, it was shown in \cite[Section 5]{holmes} that the closed subscheme $\sigma^{\lozenge -1} (E)$ of $\calM^{\lozenge}_{g,n}$ is \emph{proper} over $\overline{\calM}_{g,n}$.

Now the class $\sigma^{\lozenge *}[E]$ is by definition a Chow class on $\sigma^{\lozenge -1} (E)$ (c.f. \cite[Chapter 6]{fulton}, \cite[Definition 3.10]{Vistoli1989Intersection-th}). Since the latter is proper over $\overline{\calM}_{g,n}$, we can then push this class forward to $\overline{\calM}_{g,n}$. We define
\begin{equation} \label{deflozenge}
[DR^\lozenge] := \pi^{\lozenge}_*\left(\sigma^{\lozenge *} [E]\right).
\end{equation}
From \cite[Theorem~1.3]{holmes} we obtain when $k=0$ the equality  of Chow classes
\begin{equation} \label{holmes} [DR^\lozenge] = [DR_{LGV}]. \end{equation}

%%%%%%%%%%%%%%%%%%%%%%%%%%%%%%%
\subsection{Review of Kass--Pagani's work on $\phi$-stability} \label{kasspagani}
%%%%%%%%%%%%%%%%%%%%%%%%%%%%%%%

We first review the definition of the stability space $V_{g,n}^0$ from \cite[Definition~3.2]{kasspa2} and the notion of degenerate elements therein.  An element $\phi \in V_{g,n}^0$ is an assignment, for every  stable $n$-pointed curve $(C, p_1, \ldots, p_n)$ of genus $g$ and every irreducible component $C' \subseteq C$, of a real number $\phi(C, p_1, \ldots, p_n)_{C'}$  such that \[ \sum_{C' \subseteq C} \phi(C, p_1, \ldots, p_n)_{C'} = 0\] and such that
\begin{enumerate}
\item if $\alpha \colon (C, p_1, \ldots, p_n) \to (D, q_1, \ldots, q_n)$ is a homeomorphism of pointed curves, then $\phi(D, q_1, \ldots, q_n)= \phi(\alpha(C, p_1, \ldots, p_n))$; %the bijection that $\alpha$ induces on the irreducible components of $C$ and of $C'$ identifies $\phi(C, p_1, \ldots, p_n)$ with $\phi(C', p_1', \ldots, p_n')$;
\item informally, the assignment $\phi$ is compatible with degenerations of pointed curves.
\end{enumerate}

The notion of $\phi$-(semi)stability  was introduced in \cite[Definition~4.1, Definition~4.2]{kasspa2}:

\begin{df} \label{semistab}
Given $\phi \in V_{g,n}^0$ we say that a family $F$ of rank~$1$ torsion-free sheaves of degree~$0$ on a family of stable curves is \emph{$\phi$-(semi)stable} if the inequality
\begin{equation} \label{eqnsemistab}
		\left| \deg_{C_0}(F)- \sum \limits_{C' \subseteq C_0} \phi(C, p_1, \ldots, p_n)_{C'} + \frac{\delta_{C_0}(F)}{2} \right| <  \frac{\#(C_0 \cap \overline{C_0^{c}})-\delta_{C_0}(F)}{2}  \  \text{ (resp.~$\le$).}
	\end{equation}
holds for every stable $n$-pointed curve $(C, p_1, \ldots, p_n)$ of genus $g$ of the family, and for every subcurve (\emph{i.e.}~a union of irreducible components) $\emptyset \subsetneq C_0 \subsetneq C$. Here $\delta_{C_{0}}(F)$ denotes the number of nodes $p \in C_0 \cap \overline{C_0^{c}}$ such that the stalk of $F$ at $p$ fails to be locally free.

A stability parameter $\phi \in V_{g,n}^0$ is \emph{nondegenerate} when there is no $F$, no $(C, p_1, \ldots, p_n)$ and no $\emptyset \subsetneq C_0 \subsetneq C$ as above where equality occurs in Equation~\eqref{eqnsemistab}.
\end{df}

For all $\phi \in V_{g,n}^0$ there exists a moduli stack $\Jb_{g,n}(\phi)$ of $\phi$-semistable sheaves on stable curves, which comes with a forgetful map $\overline{p}$ to $\overline{\calM}_{g,n}$. When $\phi$ is nondegenerate, by \cite[Corollary~4.4]{kasspa2} the stack $\Jb_{g,n}(\phi)$ is Deligne--Mumford and $\mathbb{C}$-smooth, and the morphism $\overline{p}$ is representable, proper and flat.

%%%%%%%%%%%%%%%%%%%%%%%%%%%%%%%
\subsection{Compactified universal Jacobians containing $\calJ_{g,n}^{\underline{0}}$} \label{pertzero}
%%%%%%%%%%%%%%%%%%%%%%%%%%%%%%%

For some stability parameters  $\phi \in V_{g,n}^0$ the corresponding compactified universal Jacobian $\Jb_{g,n}(\phi)$ contains the  multidegree zero universal Jacobian $\calJ_{g,n}^{\underline{0}}$:
\begin{df} \label{inclusion} A nondegenerate stability parameter  $\phi \in V_{g,n}^0$ is a  \emph{small perturbation of $\underline{0}$} when the inclusion $\calJ_{g,n}^{\underline{0}} \subseteq \Jb_{g,n}(\phi)$ holds.\end{df}

Following Definition~\ref{semistab} we  explicitly characterize the small perturbations of $\underline{0}$ in $V_{g,n}^0$.
\begin{co} \label{corsmallperturb} A nondegenerate $\phi \in V_{g,n}^0$ is a small perturbation of $\underline{0}$ if and only if for every stable $n$-pointed curve $(C,p_1, \ldots, p_n)$  of genus $g$   and every subcurve $\emptyset \subsetneq C_0 \subsetneq C$, the inequality
\begin{equation} \label{smallperturb}
\left|\sum_{C' \subseteq C_0}  \phi(C, p_1, \ldots, p_n)_{C'}\right| < \frac{\# C_0 \cap \overline{C_0^{c}}}{2}
\end{equation}
holds.
\end{co}
\begin{proof}
This follows from Definition~\ref{semistab} after observing that  $\phi$ is a nondegenerate small perturbation of $\underline{0}$ if and only if the trivial line bundle is $\phi$-stable.
\end{proof}
 By \cite[Section 5]{kasspa2} the degenerate locus of $V_{g,n}^0$ is a locally finite hyperplane arrangement (because we are assuming $n \geq 1$ throughout). By applying Corollary~\ref{corsmallperturb}, we deduce that the nondegenerate small perturbations of $\underline{0}$ form a nonempty open subset of $V_{g,n}^0$.

%%%%%%%%%%%%%%%%%%%%%%%%%%%%%%%
\subsection{Extensions of the double ramification cycle as a pullback of $w^0_0$} \label{drphi}
%%%%%%%%%%%%%%%%%%%%%%%%%%%%%%%

First we extend the Brill--Noether locus $W_0^0$ defined inside $\calJ_{g,n}^0$, as a Chow class $w^0_0$ on $\Jb_{g,n}(\phi)$. Because we are assuming $n \geq 1$, by combining \cite[Corollary~4.3]{kasspa2} and \cite[Lemma~3.35]{kasspa1} we deduce the existence of a tautological family $F_{\text{tau}}$ of rank $1$ torsion-free sheaves on the total space of the universal curve $\widetilde{q} \colon \Jb_{g,n}(\phi) \times_{\overline{\calM}_{g,n}} \overline{\calC}_{g,n} \to \Jb_{g,n}(\phi)$. We define the Brill--Noether class $w(\phi)$ as
\begin{equation} \label{w00}
w(\phi) = w^0_0(\phi) :=  c_g (-\mathbb{R} \widetilde{q}_* ( F_{\text{tau}}(\phi))).
\end{equation}
We will later see in Lemma~\ref{expected} that the class $w(\phi)$ is supported on the Brill--Noether locus
\begin{equation} \label{W00}
 W(\phi) = W^0_0(\phi) := \{(C, p_1, \ldots, p_n, F) : \ h^0(C, F) >0  \} \subset \Jb_{g,n}(\phi).
\end{equation}

Then, for each nondegenerate $\phi \in V_{g,n}^0$ we define the double ramification cycle to be the pullback of $w(\phi)$ via the correspondence induced by the rational map $\sigma \colon \overline{\calM}_{g,n} \dashedrightarrow \Jb_{g,n}(\phi)$. More explicitly
\begin{equation} \label{fancypullback}
[DR(\phi)] := \sigma^*(w(\phi)) = \overline{p}_*\left([\overline{\Sigma}(\phi)] \cdot  w(\phi)\right),
\end{equation}
where $\overline{\Sigma}(\phi)$ is the closure in $\Jb_{g,n}(\phi)$ of the image of the section $\sigma$ and $\overline{p}$ is the forgetful morphism.

%%%%%%%%%%%%%%%%%%%%%%%%%%%%%%%%%%%%%%%%%%
\section{Main Result} \label{mainresult}
%%%%%%%%%%%%%%%%%%%%%%%%%%%%%%%%%%%%%%%%%

When $\phi$ is a nondegenerate small perturbation of $\underline{0}$ the approaches of Holmes and of Kass--Pagani can be directly compared. This will produce the main result of this paper.

\begin{tm} \label{main} For $\phi \in V_{g,n}^0$ a nondegenerate small perturbation of $\underline{0}$, we have the equality of classes $[DR(\phi)] = [DR^\lozenge]$.
\end{tm}
Before proving the main result we prove some preparatory lemmas.

\begin{lm} \label{expected} For $\phi \in V_{g,n}^0$ nondegenerate, the class $w(\phi)$ is  supported on the locus $W(\phi)$.
If we additionally assume  that $W(\phi)$ is irreducible, then $w(\phi) = [W(\phi)]$.
\end{lm}
\begin{proof}
	This follows from a description of $w(\phi)$ as a degeneracy class together with general results about determinental subschemes (as developed in e.g.~\cite[Section~14.4]{fulton}).  Fix a $2$-term complex $d \colon \mathcal{E}_0 \to \mathcal{E}_1$ of vector bundles that represents $\mathbb{R}\widetilde{q}_* (F_{\text{tau}})$. (Such a complex can be constructed in an elementary manner using a fixed divisor $H$ on $\Jb_{g,n}(\phi) \times_{\overline{\calM}_{g,n}} \overline{\calC}_{g,n}$  that is sufficiently $\widetilde{q}$-relatively ample.  The sheaf $F_{\text{tau}}(\phi)$ fits into a short exact sequence $0 \to F_{\text{tau}}(\phi)
	\to F_{\text{tau}}(\phi) \otimes \mathcal{O}(H) \to F_{\text{tau}}(\phi) \otimes \mathcal{O}_{H}(H) \to 0$.  The (nonderived) direct image $\widetilde{q}_{*}F_{\text{tau}}(\phi) \otimes \mathcal{O}(H) \to \widetilde{q}_{*}F_{\text{tau}}(\phi) \otimes \mathcal{O}_{H}(H)$ is a complex with the desired properties.)   We have $w(\phi) = c_{g}(\mathcal{E}_{1}-\mathcal{E}_{0})$ by definition (the $2$-term complex represents the derived pushforward appearing in Equation~\eqref{w00}), and this Chern class equals the degeneracy class of $d$  (or rather its image in the Chow group of $\Jb_{g, n}(\phi)$) by \cite[Theorem~14.4(a)]{fulton}.
	
	Since the complex $d \colon \mathcal{E}_0 \to \mathcal{E}_1$ represents $\mathbb{R}\widetilde{q}_* (F_{\text{tau}})$, it  computes the cohomology of $F_{\text{tau}}$, and this property persists after making an arbitrary base change $T \to \Jb_{g, n}$ by a $\mathbb{C}$-morphism out of a $\mathbb{C}$-scheme $T$.  Taking $T \to \Jb_{g,n}(\phi)$ to be the inclusion of a closed point $(C, p_1, \ldots, p_n, F)$, we see that $h^{0}(C, F) \ne 0$ if and only the maximal minors of $d$ vanish.  In other words, the top degeneracy subscheme $D(\phi)$ of $d \colon \mathcal{E}_{0} \to \mathcal{E}_{1}$ has support equal to  $W(\phi)$.  Being the degeneracy Chow class,  $w(\phi)$ is supported on $D(\phi)$ by construction.
	
	To complete the proof, we assume $W(\phi)$ is irreducible and then prove $w(\phi) = [W(\phi)]$.  The closure of  $\{ (C, p_1, \dots, p_n, \calO_{C}) \colon C \text{ is smooth}\}$ is an irreducible component of $W(\phi)$, so by assumption, it must equal $W(\phi)$.  An elementary computation shows that this locus has the expected codimension of $g$, so we conclude by \cite[Theorem~14.4(c)]{fulton} that  $D(\phi)$ is Cohen--Macaulay with fundamental class equal to $w(\phi)$.  Furthermore, the fiber of $D(\phi)$ over a point of $\calM_{g, n} \subset \overline{\calM}_{g,n}$  is a single reduced point (by e.g.~\cite[Proposition~4.4]{arbarello} as the fiber is a Brill--Noether locus).  Taking the point to be the generic point, we conclude that $D(\phi)$ is generically reduced and hence, by the Cohen--Macaulay condition, reduced.  Since $D(\phi)$ and $W(\phi)$ have the same support, we must have $D(\phi) = W(\phi)$ and $ w(\phi)=[W(\phi)] $.
\end{proof}

\begin{rmk} For $\phi \in V_{g,n}^0$ the Brill--Noether locus $W(\phi)$ can fail to be irreducible. Arguing as in the proof of Lemma~\ref{expected}, the closure of $\{ (C, p_1, \dots, p_n, \calO_{C}) \colon C \text{ is smooth}\}$ is an irreducible component of $W(\phi)$ of the expected dimension. Let $\Delta_{i,S}$ denote the locus of curves
having a genus $i$ component with the marked points indexed by $S$. We claim that for each boundary divisor $\Delta_{i,S} \subset \overline{\calM}_{g,n}$, there exists a nondegenerate $\phi \in V_{g,n}^0$ such that $W(\phi)$ contains the preimage of $\Delta_{i,S}$ in $\Jb_{g,n}(\phi)$. Because this preimage has codimension $1$ and is supported on the boundary, we deduce that $W(\phi)$ fails to be irreducible for this $\phi$.

We now prove the claim. By applying \cite[Proposition~3.10]{kasspa2} we deduce that, for each boundary divisor $\Delta_{i,S} \subseteq \overline{\calM}_{g,n}$ and for each $t \in \mathbb{Z}$, there exists a nondegenerate $\phi$ such that, on a pointed curve $(C, p_1, \ldots, p_n)$ that represents a point in the interior of $\Delta_{i,S}$,  all line bundles of bidegree $(t, -t)$ are $\phi$-stable. Taking $t \geq i+1$, we argue that a bidegree $(t, -t)$ line bundle $L$ admits a nonzero global section as follows.  The restriction $L|_{C_1}$ admits a nonzero section vanishing at the node by the  Riemann--Roch formula (as $C_1$ the component of $C$ of genus $i$). Prolonging this section to zero on the component $C_2$ of $C$ genus $g-i$, we produce a nonzero global section of $L$ on $C$.
\end{rmk}

We continue with more preparatory lemmas. 
 
\begin{lm} \label{restriction} For $\phi \in V_{g,n}^0$ a nondegenerate small perturbation of $\underline{0} \in V_{g,n}^0$,  we have $W(\phi) \subseteq J_{g,n}^{\underline{0}}$.
\end{lm}

This result was given by Dudin in \cite[Lemma~3.1]{dudin} in a slightly different formalism, with essentially the same proof. 

\begin{proof} Let $(C, p_1, \ldots, p_n, F)$ be in $W(\phi)$. Assume that the multidegree of $F$ is different from $\underline{0}$ and consider $s \in H^0(C, F)$. We aim to prove that $s=0$.

Because the total degree of $F$ is $0$ and the multidegree of $F$ is non-trivial, the section $s$ vanishes identically on some irreducible component of $C$. Let $C_0 \neq C$ be the (possibly empty) complement of the support of $s$.  Because the number of zeroes of a nonzero section is a lower bound on the degree of the corresponding sheaf, we deduce the inequality
\begin{equation} \label{one}
\deg_{C_0} F \geq \#C_0 \cap \overline{C_0^c}.
\end{equation}
On the other hand, if $C_0$ is nonempty,  Inequality~\eqref{eqnsemistab} for $F$ (the $\phi$-stability inequality for $F$) on $(C,p_1, \ldots, p_n)$ and $C_0 \subsetneq C$ reads
\begin{equation} \label{two}
\left|\deg_{C_0} F + \frac{\delta_{C_0}(F)}{2}-  \sum_{C' \subseteq C_0}  \phi(C, p_1, \ldots, p_n)_{C'}\right| < \frac{\# C_0 \cap \overline{C_0^c}- \delta_{C_0}(F)}{2}
\end{equation}

Combining \eqref{two} with Corollary~\ref{corsmallperturb} produces
\begin{equation} \label{three}
\deg_{C_0}  F  < \# C_0 \cap \overline{C_0^c} -  \delta_{C_0}(F).
\end{equation}

Since it is not possible for \eqref{one} and \eqref{three} to be simultaneously true (because $\delta_{C_0}(F)$ is a natural number), we deduce that $C_0 = \emptyset$ or, equivalently, that $s=0$ on $C$.
\end{proof}

The following is probably a well-known fact, but we provide a proof for the sake of completeness.
\begin{lm} \label{trivial} A line bundle $L$ of multidegree zero on a nodal curve $C$ has a nonzero global section if and only if $L$ is isomorphic to  $\mathcal{O}_C$.
\end{lm}
\begin{proof} The interesting part is the only if. Let $s$ be a  nonzero global section and $C_0$ the (nonempty) support of $s$. Consider the short exact sequence
\begin{equation} \label{blah}
0 \rightarrow \mathcal{O}_{C_0} \xrightarrow{\cdot s|_{C_0}} L|_{C_0} \rightarrow \operatorname{Coker} \rightarrow 0.
\end{equation}
defining the sheaf $\operatorname{Coker}$. Taking Euler characteristics in \eqref{blah}, we deduce $\chi(\operatorname{Coker}) = 0$ and since $\operatorname{Coker}$ is supported on points, we deduce that $\operatorname{Coker}$ is trivial. Because $s$ vanishes identically on $C_0^c$, we deduce that $C_0 = C$ and that multiplication by $s$ gives an isomorphism $\mathcal{O}_{C} \to L$.
\end{proof}

The following is an immediate consequence of the three lemmas we have proved so far.
\begin{co} \label{coro} For $\phi \in V_{g,n}^0$ a nondegenerate small perturbation of $\underline{0}$, we have $w(\phi) = [E]$ for $E$ the image of the zero section in $\Jb_{g,n}(\phi)$.
\end{co}
\begin{proof} By Lemmas~\ref{restriction} and~\ref{trivial} we deduce $W(\phi) = E$. Because $E$ is irreducible, the claim is obtained by applying Lemma~\ref{expected}.
\end{proof}

We set up some notation which we will need in the proof of Theorem~\ref{main}. Recall from Section~\ref{Section: HolmesWork} that the Abel--Jacobi section $\sigma$ is only a rational map, and $\calM_{g,n}^{\lozenge} \to \overline{\calM}_{g,n}$ is defined so that the pullback of $\sigma$ to $\calM_{g,n}^{\lozenge}$ extends.  These morphisms fit into the following pullback square defining $\calJ^{\lozenge}_{g,n}$:
\begin{equation}
\xymatrix{\calJ^{\lozenge}_{g,n} \ar^{\widetilde{\pi}^{\lozenge}}[r] \ar_{p^{\lozenge}}[d] & \calJ^{\underline{0}}_{g,n} \ar_{p}[d]  \\
\calM_{g,n}^{\lozenge} \ar^{\pi^{\lozenge}}[r] \ar@/_1pc/[u]_{e^{\lozenge}} \ar@/^2pc/[u]^{\sigma^{\lozenge}} & \overline{\calM}_{g,n}. \ar@/^1pc/@{.>}[u]^{\sigma} \ar@/_1pc/[u]_{e}}
\end{equation}
Denote by $E$ the scheme-theoretic image of $e$ and similarly with  $E^{\lozenge}$ and $\Sigma^{\lozenge}$.   Denote by $\Sigma$ the Zariski closure of the scheme-theoretic image of $\sigma$ (on the largest open substack of $\overline{\calM}_{g,n}$ where it extends to a well-defined morphism).

\begin{lm} \label{normal} The restriction of $\widetilde{\pi}^{\lozenge}$ to $\Sigma^{\lozenge}$ is the normalization $\Sigma^{\lozenge} \to \Sigma$.
\end{lm}
\begin{proof} Let $\widetilde{\Sigma}$ be the normalization of ${\Sigma}$. Note that $\pi^\lozenge$ is an isomorphism over $\calM_{g,n}$, so $\sigma^{\lozenge}$ and $\sigma \circ \pi^{\lozenge}$ coincide there. Since $\calM_{g,n}$ is schematically dense in $\calM_{g,n}^{\lozenge}$ we see that the map $\Sigma^\lozenge \to \calJ^{\underline{0}}_{g,n}$ factors through $\Sigma$. Hence by the normality of $\calM_{g,n}^{\lozenge}$ and the universal property of the normalization we get a map $\Sigma^\lozenge \to \widetilde{\Sigma}$.

Conversely, the projection map $t\colon \widetilde \Sigma \to \overline{\calM}_{g,n}$ is \emph{$\sigma$-extending} as in Definition~\ref{sigmaext}; in other words, $\widetilde \Sigma$ is normal and the rational map $\sigma\colon \widetilde \Sigma \dashedrightarrow \calJ^{\underline{0}}_{g,n}$ evidently extends to a morphism. By the universal property of $\calM_{g,n}^{\lozenge}$ we obtain a map $\widetilde \Sigma \to \calM_{g,n}^{\lozenge}$, and this map factors via the closed immersion $\Sigma^\lozenge \to \calM_{g,n}^{\lozenge}$ because $\calM_{g,n}$ is schematically dense in $\widetilde \Sigma$ and the spaces coincide over $\calM_{g,n}$.

We thus have maps $\Sigma^\lozenge \to \widetilde{\Sigma}$ and $\widetilde \Sigma \to \Sigma^\lozenge$. Moreover, both spaces are separated over $\overline{\calM}_{g,n}$, and the maps are mutual inverses over the schematically dense open $\calM_{g,n}$, hence they are mutual inverses everywhere.
\end{proof}

\begin{lm} \label{schematic} The schematic intersection of the sections $\Sigma^{\lozenge}$ and $E^{\lozenge}$ in $\calJ^{\lozenge}_{g,n}$ is proper over $\overline{\calM}_{g,n}$.\end{lm} \begin{proof} As a consequence of Lemma~\ref{normal}, the restriction of $\widetilde{\pi}^{\lozenge}$ induces an isomorphism from  the scheme-theoretic intersection of the sections $\Sigma^{\lozenge}$ and $E^{\lozenge}$ in $\calJ^{\lozenge}_{g,n}$ to the fiber product over $\calJ^{\underline{0}}_{g,n}$ of $E$ and of the normalization of $\Sigma$. The claim follows from the fact that $E$ is proper over $\overline{\calM}_{g,n}$.\end{proof}

We are now ready for the proof of the main result.

\begin{proof} (of Theorem~\ref{main})
To prove the theorem we pushforward the Chow class $[\Sigma^{\lozenge}] \cdot [E^{\lozenge}]$ along morphisms that are, in general, not proper. However, this class is supported on a proper subvariety (as shown in Lemma~\ref{schematic}), so this can be justified by choosing compatible compactifications of the various spaces involved, possibly after blowing up the boundaries to avoid extra intersections (apply \cite[Exercise II.7.12]{Hartshorne2013Algebraic-geome}) and then observing that the resulting cycles are independent of the chosen compactifications.

The push--pull formula applied to $\widetilde{\pi}^{\lozenge}$, together with the fact that $[E^{\lozenge}]~=~\widetilde{\pi}^{\lozenge *}[E]$ and Lemma~\ref{normal}, produces the equality of classes
\begin{equation} \label{pushpull}
\widetilde{\pi}^{\lozenge}_* (  [\Sigma^{\lozenge}] \cdot [E^{\lozenge}]) =  [{\Sigma}] \cdot [E].
\end{equation}

Taking the pushforward along $p$ of the left hand side of \eqref{pushpull} we obtain
\begin{equation} \label{lhs}
p_* \circ \widetilde{\pi}^{\lozenge}_* ( [\Sigma^{\lozenge}] \cdot [E^{\lozenge}]) = \pi^{\lozenge}_* \left( p^{\lozenge}_* \left( [\Sigma^{\lozenge}] \cdot  [E^{\lozenge}] \right) \right)= \pi^{\lozenge}_*( {\sigma}^{\lozenge *}  [E^{\lozenge}] )= [DR^\lozenge].
\end{equation}
The first equality is functoriality of the pushforward, the second equality is the push--pull formula for the section $\sigma^{\lozenge}$ and the last equality is Formula~\eqref{deflozenge}.

Taking the pushforward along $p$ of the right hand side of \eqref{pushpull} we obtain
\begin{equation} \label{rhs}
p_*([{\Sigma}] \cdot [E])= \overline{p}_*([\overline{\Sigma}(\phi)] \cdot [{E}])=\overline{p}_*( [\overline{\Sigma}(\phi)]\cdot w(\phi) )= [DR(\phi)]
\end{equation}
where $\phi$ is a nondegenerate small perturbation of $\underline{0}$, $\overline{p} \colon \overline{\calJ}_{g,n}(\phi) \to \overline{\calM}_{g,n}$ is the forgetful morphism and $\overline{\Sigma}(\phi)$ is the closure in $\Jb_{g,n}(\phi)$ of $\Sigma\subset \calJ^{\underline{0}}_{g,n}$. The first equality follows from the fact that $E$ is closed in $\overline{\calJ}_{g,n}(\phi)$. The second equality is Corollary~\ref{coro}.  The last equality is the definition of $[DR(\phi)]$, see Formula~\eqref{fancypullback}.

By combining Equations~\eqref{pushpull}, \eqref{lhs} and \eqref{rhs} we conclude
\[
[DR(\phi)]= [DR^\lozenge].
\]
\end{proof}

\begin{rmk} \label{analogous}
As an interesting by-product of the proof of Theorem~\ref{main} we also obtain a simple description of $[DR^\lozenge]$ (and hence, when $k=0$, of the Li--Graber--Vakil extension of the double ramification cycle) as
\begin{equation} \label{classes}
[DR^\lozenge]= p_* ([\Sigma] \cdot [E])= e^*([\Sigma])
\end{equation}
for $p \colon \calJ^{\underline{0}}_{g,n} \to \overline{\calM}_{g,n}$ the natural forgetful morphism. By the definition of pullback along the rational map $\sigma$ (see Equation~\eqref{fancypullback}), the common class in~\eqref{classes} can also be described as $\sigma^*([E])$. % and $\Sigma$ the scheme-theoretic image of $\sigma|_{\calM_{g,n}} \colon \calM_{g,n} \to \calJ^{\underline{0}}_{g,n}$.
\end{rmk}

%%%%%%%%%%%%%%%%%%%%%%%%%%%%%%%%%%%%%%%%%%
\section{Consequences}\label{consequences}
%%%%%%%%%%%%%%%%%%%%%%%%%%%%%%%%%%%%%%%%%%

In \cite[Section~6.1]{kasspa2} the authors characterized the set of nondegenerate $\phi \in V_{g,n}^0$ with the property that the universal line bundle $\omega_C^{-\otimes k}(a_1 p_1 + \ldots + a_n p_n)$ is $\phi$-stable. For such $\phi$'s, Formula~\eqref{fancypullback} reduces to the usual pullback $\sigma^*(w(\phi))$  by the lci morphism $\sigma$  and the corresponding extension of the double ramification cycle is computed as %precisely as in Formula~\eqref{formulaopen}
\begin{equation} \label{formulaclosed}
 [DR(\phi)] = c_g \left( - \mathbb{R}q_* \left(\omega_C^{-\otimes k}(a_1 p_1 + \ldots + a_n p_n) \right) \right)
 \end{equation}
 The computation is derived by using the  definition of $w(\phi)$ in Formula~\eqref{w00}, invoking  cohomology and base change, and then applying the Grothendieck--Riemann--Roch formula to the universal curve  $q \colon \overline{\calC}_{g,n}\to\overline{\calM}_{g,n}$ as in \cite[Part~II]{mumford}.

All other classes $[DR(\phi)]$ can in principle be computed by applying wall-crossing formulae (as carried out in \cite[Theorem~4.1]{kasspa1} by Kass--Pagani in the similar but simpler case of the theta divisor, the Brill--Noether class $w^r_d$ with $r=0$ and $d=g-1$). As we mentioned in the introduction, this gives a new approach to computing the class of the double ramification cycle --- either $[DR_{LGV}]$ when $k=0$, or for general $k$ the cycle $[DR^\lozenge]$, which conjecturally agrees with Pixton's formula, see \cite[Conjecture~1.4]{holmes}.

A natural question at this point is whether it is possible for some universal line bundle $\omega_C^{-\otimes k}(a_1 p_1 + \ldots + a_n p_n)$ to be $\phi$-stable \emph{for some nondegenerate small perturbation of $\underline{0}$}. This happens only when the vector $(k; a_1, \ldots, a_n)$ is \emph{trivial}, \emph{i.e.}~when $k(2-2g)=a_1= \dots=a_n=0$. Indeed, if $\phi$ is nondegenerate, then on curves with $1$ separating node there is a unique $\phi$-stable bidegree of line bundles by Definition~\ref{semistab}. If $\phi$ is a small perturbation of $\underline{0}$, this bidegree must be $(0,0)$. Therefore to be $\phi$-stable, the universal line bundle $\omega_C^{-\otimes k}(a_1 p_1 + \ldots + a_n p_n)$ must have trivial bidegree on all curves with $1$ separating node, which implies that it is trivial.

A better question is to ask if it is possible  that, for some \emph{nontrivial} vector $(k; a_1, \ldots, a_n)$, the corresponding Abel--Jacobi section $\sigma= \sigma_{k; a_1, \ldots, a_n}$ extends to a well-defined morphism $\overline{\calM}_{g,n} \to \Jb_{g,n}(\phi)$ for some nondegenerate small perturbation $\phi$ of $\underline{0}$. This happens only for the vectors $(k; a_1, \ldots, a_n)$ that are very close to $\underline{0}$, in a sense that we make precise in the following proposition.

\begin{pr} \label{cor} Let $g, n\geq 1$ and assume $(k; a_1, \ldots, a_n)$ is not trivial. The corresponding Abel--Jacobi section $\sigma$ extends to a well-defined morphism $\overline{\calM}_{g,n} \to \Jb_{g,n}(\phi)$ for some  nondegenerate small perturbation $\phi$ of $\underline{0}$ if and only if $k(2-2g)=0$ and $\underline{a}= (0, \ldots, \pm 1, \ldots, \mp 1, \ldots, 0)$.
\end{pr}
\begin{proof}

For simplicity we only discuss the case $g \geq 2$ (the case $g=1$ is similar and simpler).

To prove our claim we invoke \cite[Corollary~6.5]{kasspa2}, which implies that $\sigma= \sigma_{k; a_1, \ldots, a_n}$ extends to a well-defined morphism $\overline{\calM}_{g,n}\to \Jb_{g,n}(\phi)$ if and only if the universal line bundle $\omega_C^{-\otimes k}(a_1 p_1 + \ldots + a_n p_n)$ is $\phi$-stable on all stable pointed curves $(C, p_1, \ldots, p_n)$ that consist of $2$ smooth irreducible components meeting in at least $2$ nodes.

Assume $k=0$ and $\underline{a}= (0, \ldots, a_i= 1, \ldots, a_j=- 1, \ldots, 0)$ and define $\phi \in V_{g,n}^0$ using \cite[Isomorphism~(11)]{kasspa2} to be the unique stability parameter that is trivial over all stable curves with $1$ separating node (in the notation of \cite{kasspa2}, its projection to $C_{g,n}$ is trivial) and such that \[\phi(\Gamma_i) = \left( \frac{1}{2} + \epsilon_i, -\frac{1}{2} - \epsilon_i\right), \quad \phi(\Gamma_j) = \left( -\frac{1}{2}- \epsilon_j, \frac{1}{2}+\epsilon_j\right), \quad \phi(\Gamma_{k \neq i, j}) = \left(\epsilon_k, - \epsilon_k\right)\] for some perturbation $0 < || (\epsilon_1, \ldots, \epsilon_n)||<<1$ making the parameter $\phi$ nondegenerate. (Here $\Gamma_t$ for $t=1, \ldots, n$ is any curve with a smooth component of genus $0$ carrying the marking $p_t$, connected by $2$ nodes to a smooth component of genus $g-1$ with all other markings).  To check that $\phi$ is a small perturbation of $\underline{0}$, by \cite[Corollary~5.9]{kasspa2} it is enough to show that the trivial line bundle is $\phi$-stable over all curves with $2$ smooth irreducible components, which is achieved by applying \cite[Formula~(29)]{kasspa2}. To conclude we prove that $\mathcal{O}_C(a_1 p_1 + \ldots + a_n p_n)$ is $\phi$-stable on every stable pointed curve $(C, p_1, \ldots, p_n)$ that consists of $2$ smooth irreducible components and at least $2$ nodes.  By applying \cite[Formula~(29)]{kasspa2} we deduce the inequality
\begin{equation}
  \left|\sum_{i : p_i \in C'} a_i   - \phi(C,p_1, \ldots, p_n)_{C'} \right| < \frac{  \#\text{Sing} (C)}{2},
\end{equation}
where  $C'$ denotes either of the components of $C$. By Definition~\ref{semistab} we have that  $\mathcal{O}_C(a_1 p_1 + \ldots + a_n p_n)$ is $\phi$-stable on $(C,p_1, \ldots, p_n)$ and we conclude that $\sigma$ extends to a well-defined morphism on $\overline{\calM}_{g,n}$.

For the other implication we use the following criterion. By Definition~\ref{semistab}, if $(C, p_1, \ldots, p_n)$ is a stable curve that consists of $2$ smooth irreducible components meeting in $2$ nodes and $\phi$ is a nondegenerate small perturbation of $\underline{0}$, then a line bundle of bidegree $(t,-t)$ is $\phi$-stable if and only if $t=\pm 1$ (because the $\phi$-stable bidegrees are $2$  consecutive bidegrees and one of them is $(0,0)$).

The universal line bundle $\omega^{-k}(a_1 p_1 + \ldots + a_n p_n)$ has bidegree $(-2k,2k)$ on  the stable curve that consists of a smooth component of genus $1$ without markings connected by $2$ nodes to a smooth component of genus $g-2$ with all markings and by the criterion we explained above the universal line bundle cannot be $\phi$-stable unless $k=0$. Assuming now $k=0$, if $\underline{0} \neq \underline{a}\neq (0, \ldots, \pm 1, \ldots, \mp 1, \ldots, 0)$ there are $1 \leq i \leq j \leq n$ such that $a_i + a_j=t \geq 2$. The universal line bundle has bidegree $(t,-t)$ on the stable curve that consists of a smooth component of genus $0$ with the markings $p_i$ and $p_j$, connected by $2$ nodes to a smooth component of genus $g-1$ with all other markings. By applying the criterion again, the universal line bundle is not $\phi$-stable and the proof is concluded.% By Definition~\ref{semistab} there are exactly $2$ consecutive $\phi$-stable bidegrees on this curve, so there exists no nondegenerate small perturbation $\phi$ of $\underline{0}$ such that $\calO(a_1 p_1 + \ldots + a_n p_n)$ is $\phi$-stable.
\end{proof}

The proposition makes it possible,  when $(k; a_1, \ldots, a_n)$ is nontrivial and very close to~$\underline{0}$, to describe the class $[DR_{LGV}]$ as a degree-$g$ Chern class similar to Formula~\eqref{formulaclosed}. By \cite[Proposition~6.4]{kasspa2} the map $\sigma$ extends to a well-defined morphism $\overline{\calM}_{g,n} \to \Jb_{g,n}(\phi)$ if and only if the universal line bundle $\ODF$ is $\phi$-stable. Here $\ODF$ is the unique universal line bundle satisfying the following two conditions:
 \begin{enumerate}
 \item the line bundles $\omega^{-k}(a_1 p_1 + \ldots + a_n p_n)$  and $\ODF$ coincide on $\calM_{g,n}$ and
 \item the line bundle $\ODF$ is $\phi$-stable on $\calM_{g,n}^{\leq 1}$, the moduli stack of stable curves with at most one node.
 \end{enumerate}
 For $(g,n)$ and $(k; a_1, \ldots, a_n)$ and $\phi \in V_{g,n}^0$ as in Proposition~\ref{cor}, by arguing along the lines of \eqref{formulaclosed}, we obtain
\begin{equation}
[DR_{LGV}]=  c_g ( - \mathbb{R}q_* (\ODF)).
\end{equation}

\bibliographystyle{BRnote}

\bibliography{biblio}

\end{document}